\documentclass[a4paper]{amsart}
\usepackage{amsmath,amssymb,amsxtra,amsfonts,amsthm,comment,url}

\theoremstyle{definition}

\usepackage{colortbl}

\theoremstyle{plain}
\newtheorem{theorem}{Theorem}
\newtheorem{proposition}{Proposition}
\newtheorem{lemma}{Lemma}
\newtheorem{remark}{Remark}
\newtheorem{corollary}{Corollary}

\usepackage{colortbl}

\begin{document}

\title[Diophantine properties in $\beta$-dynamical systems]{A dichotomy law for the Diophantine properties in $\beta$-dynamical systems}

\author{Michael Coons}
\address{School of Mathematical and Physical Sciences, University of Newcastle, Callaghan, 2308, NSW, Australia}
\email{michael.coons@newcastle.edu.au}

\author{Mumtaz Hussain}
\address{School of Mathematical and Physical Sciences, University of Newcastle, Callaghan, 2308, NSW, Australia}
\email{mumtaz.hussain@newcastle.edu.au, drhussainmumtaz@gmail.com}

\author{Bao-Wei Wang}
\address{School of Mathematics and Statistics, Huazhong University of Science and Technology, 430074 Wuhan, China}
\email{bwei\_wang@hust.edu.cn}

\keywords{Beta-expansions, Diophantine approximation, Jarn\'ik-type theorem, Hausdorff measure}
\subjclass[2010]{Primary 11J83; Secondary 11K60, 11K16}%

\thanks{The research of M.~Coons and M.~Hussain is supported by the Australian Research Council (DE140100223) and the research of B.-W.~Wang is supported by NSFC of China (No. 11471130 and NCET-13-0236).}

\begin{abstract} Let $\beta>1$ be a real number and define the $\beta$-transformation on $[0,1]$ by $T_\beta:x\mapsto \beta x\bmod 1$. Further, define $$W_y(T_{\beta},\Psi):=\{x\in [0, 1]:|T_\beta^nx-y|<\Psi(n) \mbox{ for infinitely many $n$}\}$$ and $$W(T_{\beta},\Psi):=\{(x, y)\in [0, 1]^2:|T_\beta^nx-y|<\Psi(n) \mbox{ for infinitely many $n$}\},$$ where $\Psi:\mathbb{N}\to\mathbb{R}_{>0}$ is a positive function  such that $\Psi(n)\to 0$ as $n\to \infty$. In this paper, we show that each of the above sets obeys a Jarn\'ik-type dichotomy, that is, the generalised Hausdorff measure is either zero or full depending upon the convergence or divergence of a certain series. This work completes the metrical theory of these sets.
\end{abstract}

\maketitle

\section{Introduction}

Let $(X, T, \mu, \mathcal{B})$ be a measure-theoretic dynamical system, where $T:X\to X$ is a transformation on $X$, $\mu$ is a finite $T$-invariant Borel measure, and $\mathcal{B}$ is the associated Borel $\sigma$-algebra. The famous Poincar\'{e} recurrence theorem implies that for almost all $x\in X$, the $T$-orbit of $x$ is dense in $X$. That result is qualitative in nature, though it leads to the study of the quantitative properties of the distribution of the $T$-orbits of points in the space $X$, which is called  dynamical Diophantine approximation. More precisely, the spotlight is on the size of the set $$W_y(T,\Psi):=\left\{x\in X: |T^nx-y|<\Psi(n)\mbox{ for infinitely many $n$}\right\},$$
where $\Psi:\mathbb{N}\to \mathbb{R}_{>0}$ is a positive function such that $\Psi(n)\to 0$ as $n\to \infty$. The set $W_y(T,\Psi)$ is the dynamical analogue of the classical well-approximable set (e.g., see \cite{BaBV13, BaHH15, Do97, HK15}) and it has close connections to classic Diophantine approximation, for example when $T$ is an irrational rotation or Gauss transformation. It has been an object of significant interest since the pioneering works of Philipp \cite{Phi} on the $\mu$-measure of $W_y(T,\Psi)$ and Hill and Velani \cite{HiV95} on the Hausdorff dimension of $W_y(T,\Psi)$. It is easy to see from the definition that the set $W_y(T,\Psi)$ contains the points in $X$ whose $T$-orbit hits a shrinking target infinitely often; shrinking target problems for similar situations have been studied by Chernov and Kleinbock \cite{ChK}, Hill and Velani \cite{HiV95, HiV99}, and Tseng \cite{Ts08} among others.

When the system $(X, T, \mu, \mathcal{B})$ possesses strong mixing properties, similar to Khintchine's theorem and its generalisations in classical Diophantine approximation, the $\mu$-measure of $W_y(T,\Psi)$ is zero or full, according to the convergence or divergence of a certain series. Philipp \cite{Phi} proved this for $b$-ary expansions, $\beta$-expansions, and continued fractions.

Properties of $W_y(T,\Psi)$ are related to the distribution of the inverse images $\{T^{-n}y\}_{n\geqslant 1}$ of $y$. If the $\mu$-measure of $W_y(T,\Psi)$ obeys a dichotomy law, it means, in some sense, that  $\{T^{-n}y\}_{n\geqslant 1}$ is regularly distributed. In this way, one expects that tools from the theory of metric Diophantine approximation, such as regular systems \cite{BaS}, ubiquitous systems \cite{BeDo_99, DodV}, and the mass transference principle \cite{BV06}, can be used to derive the size of $W_y(T,\Psi)$ in terms of Hausdorff measure. More precisely,  similar to the Jarn\'ik dichotomy law, one expects that there should be a dichotomy law (zero or full) for the Hausdorff measure of the dynamically defined limsup set $W_y(T,\Psi)$.

Following the work of Hill and Velani \cite{HiV95, HiV99}, the Hausdorff dimension of the set $W_y(T, \Psi)$ has been  determined in many systems, from the system of rational expanding maps on their Julia sets to systems with non-finite Markov systems \cite{ShWa} and conformal iterated function systems \cite{LiW, Rev, Urb}. However, the Hausdorff measure of $W_y(T, \Psi)$ is currently known only for systems with finite Markov properties \cite{BerDV, HiV02}. We remedy this situation.

In this paper, we consider the Hausdorff measure of $W_y(T, \Psi)$ on $\beta$-expansions.
There are two reasons that we choose to consider this non-finite Markov system.
On the first hand, combining with Philipp's work, we hope to provide a complete metric theory on the size of $W_y(T, \Psi)$. Moreover, the non-finite Markov property for $\beta$-expansions remains a barrier to determining metric properties, so we want to see whether new ideas will be found in considering this concrete question. 
On the other hand, when given a full Lebesgue measure statement, the mass transference principle has proven a powerful tool in studying the Hausdorff measure of a limsup set in classic Diophantine approximation \cite{BV06} as well as dynamical Diophantine approximation for systems with finite Morkov properties \cite{BerDV}. But to the authors' knowledge, it seems that there are exceptions. For example, there is a full Lebesgue measure statement \cite{Kim} on the size of the limsup set $$
    \Big\{y: |n\alpha-y|<\psi(n) \ {\text{for infinitely many}}\ n\in \mathbb{N}\Big\},
    $$ but we do not think a direct application of mass transference principle would give even the right Hausdorff dimension, let alone its Hausdorff measure. For this non-finite Markov system, we have to give some modifications on the mass transference principle and also need to carefully choose a subset of $W_y(T, \Psi)$ to make the mass transference principle applicable.

Now let's focus on the $\beta$-expansion. For a real number $\beta>1$, define the transformation $T_\beta:[0,1]\to[0,1]$ by $$T_\beta: x\mapsto \beta x\bmod 1.$$ This map generates the $\beta$-dynamical system $([0,1], T_\beta)$. It is well known that $\beta$-expansion is a typical example of an expanding non-finite Markov system whose properties are reflected by the orbit of some critical point; here, it is the expansion of 1. General $\beta$-expansions have been widely studied in the literature, beginning with the pioneering works of R\'enyi \cite{Re57} and Parry \cite{Pa60}, and continuing with Hofbauer~\cite{Hof}, Persson and Schmeling \cite{PeSch}, Schmeling \cite{Schme}, and Tan and Wang \cite{TaW} to name just a few.

We are interested in the size of the dynamically defined limsup set
\begin{equation}\label{WyTbP}W_y(T_{\beta}, \Psi):=\left\{x\in [0,1]:|T_\beta^nx-y|<\Psi(n)\mbox{ for infinitely many $n$}\right\},\end{equation} where, as above, $\Psi:\mathbb{N}\to \mathbb{R}_{>0}$ is a positive function. Philipp \cite{Phi} showed that the Lebesgue measure or Parry measure of the set $W_y(T_{\beta}, \Psi)$ is zero or full according to the convergence or divergence of the series $\sum_{n\geqslant 1}\Psi(n)$. The Hausdorff dimension of $W_y(T_{\beta},\Psi)$ was given by Shen and Wang \cite{ShWa} (see also Bugeaud and Wang \cite{BuWa14}). As stated above, in this paper, we focus on the Hausdorff measure of $W_y(T_{\beta},\Psi)$.

Throughout this paper, a  \emph{dimension function}  is a function $f:\mathbb{R}\to \mathbb{R}$ such that $f(r)\to 0$ as $r\to 0$ and such that $f$ is increasing in $[0,r_0)$ for some $r_0>0$; $\mathcal{H}^f$ denotes the $f$-dimensional Hausdorff measure. For the definitions of Hausdorff dimension and Hausdorff measure, we refer to the standard texts by Bernik and Dodson \cite{BeDo_99} and Falconer \cite{F_03}. In this paper, we prove the following dichotomy law for $\beta$-dynamical systems, which is an analogue of classical Jarn\'ik-type theorems. It is worth noting that our results are the first concerning Hausdorff measures of $\beta$-dynamical systems.

\begin{theorem}\label{thm1} Let $\Psi:\mathbb{N}\to\mathbb{R}_{>0}$ be a positive function. Let $f$ be a dimension function such that $r^{-1}f(r)$ is monotonic. For any $\beta>1$, we have $$\mathcal{H}^f\Big(W_y(T_{\beta},\Psi)\Big)=\begin{cases} 0 & \mbox{when\ \ $\displaystyle\sum_{n\geqslant 1} f\left(\frac{\Psi(n)}{\beta^n}\right) \beta^n$\ converges},\\
\vspace{-.4cm}&\\
\mathcal{H}^f([0,1]) & \mbox{when\ \ $\displaystyle\sum_{n\geqslant 1} f\left(\frac{\Psi(n)}{\beta^n}\right) \beta^n$\ diverges}.\end{cases}$$
\end{theorem}

The condition `$r^{-1}f(r)$ is monotonic' is not a particularly restrictive condition, and it is the main ingredient in unifying both the Lebesgue and Hausdorff measure statements; for details see Beresnevich and Velani \cite{BV06}. To be precise,  $\mathcal{H}^f$ is proportional to the standard Lebesgue measure when $f(r)\asymp r^1$. When $f(r)\asymp r^s$, we write $\mathcal{H}^s$ in place of $\mathcal{H}^f$, and whenever $\Psi(r)=r^{-\tau}$ for $\tau>0$, we write $W_y(T_\beta, \tau)$ in place of $W_y(T_\beta, \Psi)$.

Theorem \ref{thm1} can be further generalised by considering the set $$W(T_\beta,\Psi):=\left\{(x, y)\in [0,1]^2:|T_\beta^nx-y|<\Psi(n)\mbox{ for infinitely many $n$}\right\}.$$
This set can be viewed as the doubly metrical $\beta$-dynamical analogue of the classic Diophantine set as given by Dodson \cite{Do97}. The Hausdorff dimension of $W(T_\beta, \Psi)$ was given by Ge and L\"{u} \cite{GeLu}.
Its Hausdorff measure is given as follows.
\begin{theorem}\label{thm2} Let $\Psi:\mathbb{N}\to \mathbb{R}_{>0}$ be a positive function. Let $g$ be a dimension function such that $r^{-2}g(r)$ is monotonic. For any $\beta>1$, we have $$\mathcal{H}^g\Big(W(T_{\beta}, \Psi)\Big)=\begin{cases} 0 & \mbox{when\ \ $\displaystyle\sum_{n\geqslant 1} g\left(\frac{\Psi(n)}{\beta^n}\right) \frac{\beta^{2n}}{\Psi(n)}$\ converges}, \\
\vspace{-.4cm}&\\
\mathcal{H}^g\left([0,1]^2\right) & \mbox{when\ \ $\displaystyle\sum_{n\geqslant1} g\left(\frac{\Psi(n)}{\beta^n}\right) \frac{\beta^{2n}}{\Psi(n)}$\ diverges}.\end{cases}$$
\end{theorem}

An immediate consequences of Theorems \ref{thm1} and \ref{thm2} are not only the respective Hausdorff dimension results, but also that $\mathcal{H}^s(W_y(T_\beta, \tau))=\mathcal{H}^s([0,1])$ when $s=\dim_H W_y(T_\beta, \tau)=1/(\tau+1),$ and $\mathcal{H}^s(W(T_\beta, \tau))=\mathcal{H}^s([0,1]^2)$ when $s=\dim_H W(T_\beta, \tau)=1+1/(\tau+1).$ In general, a Hausdorff measure result is much stronger than a Hausdorff dimension result as it allows one to distinguish sets of equal Hausdorff dimension. In fact, more subtle examples can be given to reiterate the significance of each of Theorems~\ref{thm1} and \ref{thm2}. For example, regarding Theorem \ref{thm2}, for $\tau>0$, set $\Psi_1(n)=\left(\beta^n\right)^{-\tau}$ and for some $\varepsilon>0$, set $\Psi_\varepsilon(n)= \left(\beta^n\right)^{-\tau}\left(\log(\beta^n)\right)^{-(1+\varepsilon)(\tau+1)/(\tau+2)}.$ We then have the following exact logarithmic order for $\beta$-approximation.

\begin{corollary} Let $g(r)=r^{(2+\tau)/(1+\tau)}$. For any $\varepsilon>0$, 
$$\mathcal{H}^{g}(W(T_\beta, \Psi_1))=\mathcal{H}^{g}([0,1]^2) \quad\mbox{and}\quad\mathcal{H}^{g}(W(T_\beta, \Psi_\varepsilon))=0.$$ Consequently, the set $W(T_\beta, \Psi_1)\setminus W(T_\beta, \Psi_\varepsilon)$ is uncountable.
\end{corollary}


\section{Preliminaries}

In this section, we first collect some basic properties of $\beta$-expansions and fix some notation. We then state versions of Philipp's result \cite{Phi} and Beresnevich and Velani's slicing lemma \cite{BV06_IMRN}, before giving a variant of their famous mass transference principle \cite{BV06} fit for our use.

For a real number $x\geqslant 0$, we write $\lfloor x\rfloor$ for the integer part of $x$. Using the $\beta$-transformation $T_\beta$,
each $x\in[0,1]$ can be uniquely expressed as a finite or an infinite series, known as the $\beta$-expansion of $x$; see R\'enyi \cite{Re57}.
That is, for each $x\in[0,1]$, we have \begin{equation}\label{1} x=\sum_{i\geqslant 1} \frac{\epsilon_i(x,\beta)}{\beta^i},\end{equation} where $\epsilon_i(x,\beta)=\lfloor \beta T_\beta^{i-1}x\rfloor$ for each $i\geqslant 1$.
Now, for any $x\in[0,1]$ and $n\in\mathbb{N}$, by the definition of $\beta$-expansion,
\begin{equation}\label{eq5}
x=\frac{\epsilon_1(x, \beta)}{\beta}+\cdots+\frac{\epsilon_n(x, \beta)+T_\beta^nx}{\beta^n}.
\end{equation}

The $\beta$-expansion of $1$ is of significant importance. To highlight this, we define an infinite sequence related to the expansion of $1$. If the expansion of $1$ in \eqref{1} is infinite, that is, $\epsilon_n(1,\beta)\ne 0$ for infinitely many $n$, then define $$
(\epsilon^*_1, \epsilon^*_2,\ldots)=(\epsilon_1(1,\beta),\epsilon_2(1,\beta),\ldots),
$$
and if the expansion of $1$ in \eqref{1} is finite, that is $$
1=\frac{\epsilon_1(1,\beta)}{\beta}+\cdots+\frac{\epsilon_n(1,\beta)}{\beta^n}, \  {\rm with}\ \epsilon_n(1,\beta)\ne 0,
$$ then define
$$
(\epsilon^*_1, \epsilon^*_2,\ldots)=(\epsilon_1(1,\beta),\ldots, \epsilon_{n-1}(1,\beta), \epsilon_n(1,\beta)-1)^{\infty},
$$ where $w^\infty$ denotes the periodic sequence 
$(w, w,\ldots)$ for a finite word $w$.
Each of the sequences $(\epsilon^*_1, \epsilon^*_2,\ldots)$ are called the {\em infinite digit sequence} of the expansion of 1.

For each $n\in\mathbb{N}$, let $D_{\beta,n}$ denote all {\em admissible sequences} of length $n$, that is, $$
D_{\beta, n}=\left\{(\epsilon_1,\ldots, \epsilon_n)\in \mathbb{Z}_{\geqslant 0}^n: \exists x\in [0,1]\ {\rm such\ that}\ \epsilon_i(x,\beta)=\epsilon_i, 1\leqslant i\leqslant n\right\}.
$$
The characterisation of the elements in $D_{\beta, n}$ and its cardinality $\# D_{\beta, n}$ are given by Parry \cite{Pa60} and R\'enyi \cite{Re57} {in the lemma below. First recall the definition of the lexicographical order $\preceq$. We write $$(\epsilon_1, \epsilon_2, \ldots, \epsilon_n)
\preceq (\epsilon_1', \epsilon_2', \ldots, \epsilon_n')$$
 if for every $j \geqslant 1$ we have $\epsilon_j \leqslant\epsilon_j'$}
\begin{lemma}[Parry, R\'enyi] A non-negative integral word $(\epsilon_1,\ldots, \epsilon_n)$ belongs to $D_{\beta, n}$ if and only if, in the lexicographical order,
$$(\epsilon_{k+1},\ldots,\epsilon_n)\preceq (\epsilon^*_1,\ldots, \epsilon^*_{n-k}),\ {\text{for all}}\ 0\leqslant k<n.$$
Moreover,
\begin{equation}\label{lem1}\beta^n\leqslant \# D_{\beta,n}\leqslant \frac{\beta^{n+1}}{\beta-1}.\end{equation}
\end{lemma}


For each $\bar{\epsilon}=(\epsilon_1,\ldots,\epsilon_n)\in D_{\beta,n}$ with $n\geqslant1$, we define the $n$th order cylinder $I_n(\bar{\epsilon})$ by
$$I_n(\bar{\epsilon})=I_n(\epsilon_1, \ldots, \epsilon_n)=\left\{x\in[0,1]: \epsilon_i(x, \beta)=\epsilon_i\ \mbox{for all $1\leqslant i\leqslant n$}\right\}.$$  The cylinder $I_n(\bar{\epsilon})$ is a non-empty interval with left-endpoint $$
\frac{\epsilon_1}{\beta}+\cdots+\frac{\epsilon_n}{\beta^n}
$$  and with length at most $\beta^{-n}$. The exact length of a cylinder is given in \cite{FanWa}  {and it depends on the digit sequence $\epsilon_1, \ldots, \epsilon_n$.}

Before moving on to results concerning various measures and sets, we note that the interval $[0,1]$ is partitioned by the cylinders $I_n(\bar{\epsilon})$; that is, we have the disjoint union
\begin{equation}\label{eq1}
[0,1]=\bigcup_{\bar{\epsilon}\in D_{\beta,n}}I_n(\bar{\epsilon}).
\end{equation}

In the rest of this paper, we use the following notation concerning `size'. For a set $A$, we denote the Lebesgue measure of $A$ by $\mathfrak{L}(A)$, and we denote the diameter of an interval $I$ by $|I|$. Of course, for an interval $I$, we have $|I|=\mathfrak{L}(I)$. Note that we also use the notation $|\cdot|$ to denote absolute value; we believe the context of usage is unambiguous.
With this notation, we set out our more measure-theoretic preliminaries.

We start by recalling the following metrical result of Philipp \cite{Phi} concerning the Lebesgue measure of the set $W_y(T_\beta, \Psi)$ and a one-sided variant.

\begin{proposition}[Philipp]\label{p1}
Let $W_y(T_\beta, \Psi)$ be the set defined in \eqref{WyTbP}, and define $$W'_y(T_{\beta}, \Psi):=\left\{x\in [0,1]: 0\leqslant T_\beta^nx-y<\Psi(n)\ \mbox{for infinitely many $n$}\right\}.$$ Then
$$\mathfrak{L}\left(W_y(T_\beta, \Psi)\right)=\mathfrak{L}\left(W'_y(T_\beta, \Psi)\right)=\begin{cases} 0 & \mbox{when\ \ $\displaystyle\sum_{n\geqslant 1} \Psi(n)$\ converges}, \\
\vspace{-.4cm}&\\
1 & \mbox{when\ \ $\displaystyle\sum_{n\geqslant1} \Psi(n)$ diverges}.\end{cases}$$
\end{proposition}

While both parts of Proposition \ref{p1} are special cases of Philipp's result \cite{Phi}, as a sequence of intervals takes on the role of the balls $\{B(y, \Psi(n))\}_{n\geqslant 1}$, it is worth noting that the result for $W'_y(T_\beta, \Psi)$ can also be deduced from the result for $W_y(T_\beta, \Psi)$ using the Lebesgue density theorem. We cite a general result due to Cassels; see Harman \cite[Lemma 2.1]{Har}.

\begin{lemma}[Cassels]
Let $\{I_k\}_{k\geqslant 1}$ be a sequence of intervals such that $\mathfrak{L}(I_k)\to 0$ as $k\to\infty$. If $\{J_k\}_{k\geqslant 1}$ is a sequence of measurable sets such that $J_k\subseteq I_k$ for each $k\geqslant 1$, and there is a positive real number $\delta$ such that $\mathfrak{L}(J_k)\geqslant\delta\cdot\mathfrak{L}(I_k),$ then $$\mathfrak{L}\left(\limsup_{k\to \infty}J_k\right)=\mathfrak{L}\left(\limsup_{k\to \infty}I_k\right).$$
\end{lemma}

We next state a variant of the `slicing' lemma due to Beresnevich and Velani~\cite{BV06_IMRN}. This version is tailored for our use, and is a key ingredient in the proof of Theorem~\ref{thm2}. The slicing technique is broad-ranging and has been  useful in proving several metrical results; for examples, see Hussain and Kristensen \cite{HK13_2, HK15} and Hussain and Levesley \cite{HL13}.

\begin{lemma}[Beresnevich and Velani]\label{slicing}
Suppose that $g$ and $f:r\to r^{-1}g(r)$ are dimension functions. Let $B\subseteq \mathbb{R}^2$ be a Borel set and let $V$ be a $1$-dimensional linear subspace of $\mathbb{R}^2$. If there is a subset $S$ of the orthogonal complement of $V$ such that $\mathcal{H}^1(S)>0$ and for each $b\in S$, $$\mathcal{H}^{f}\left(B\cap(V+b)\right)=\infty,$$ then $\mathcal{H}^{g}(B)=\infty.$
\end{lemma}

The main ingredient in establishing Theorem \ref{thm1} is the {\em mass transference principle} of Beresnevich and Velani \cite{BV06}. Given a dimension function $f$ and a sequence of balls $B_i\subseteq \mathbb{R}$, by definition, $\limsup_{i\to\infty}B_i$ is precisely the set of points which lie in infinitely many of the balls $B_i$. Further, for a ball $B=B(x,r)$, set $B^f=B(x, f(r))$.  The following mass transference principle is tailored to suit our needs; for a general statement and further details, we refer the reader to the paper of Beresnevich and Velani \cite[Theorem 2]{BV06}.

\begin{proposition}[Mass Transference Principle]\label{p3} Let $\{B_i\}_{i\geqslant 1}$ be a sequence of balls in $\mathbb{R}$ with $|B_i|\to 0$ as $i\to \infty$ and let $f$ be a dimension function such that {$r^{-1}f(r)$ is non-decreasing as $r\to 0$}. Suppose that for any ball $B\subseteq\mathbb{R}$ \begin{equation*} \mathcal{H}^1(B\cap\limsup B_i^f)=\mathcal{H}^1(B).\end{equation*} Then for any ball $B\subseteq\mathbb{R}$, \begin{equation*} \mathcal{H}^f(B\cap\limsup B_i)=\mathcal{H}^f(B).\end{equation*}
\end{proposition}

In essence, the mass transference principle allows one to translate statements about the Lebesgue measure of general limsup sets to ones involving Hausdorff measure. So, using Proposition \ref{p1}, one should be able to say something about the Hausdorff measure of $W_y(T_{\beta}, \Psi)$. Indeed, this turns out to be the case, but we must first make some minor modifications to the mass transference principle.

\begin{proposition}[A variant of Mass Transference Principle]\label{p4} Let $\{x_n\}_{n\geqslant 1}$ be a sequence of points in $[0,1]$ and $\{r_n\}_{n\geqslant 1}$ a sequence of positive numbers with $r_n\to 0$ as $n\to \infty$. Let $f$ be a dimension function such that {$r^{-1}f(r)$ is non-decreasing as $r\to 0$}. If  \begin{equation}\label{f4} \mathfrak{L}\left(\left\{x\in [0,1]:0\leqslant x-x_n< f(r_n)\mbox{ for infinitely many $n$}\right\}\right)=1,\end{equation} then for any ball $B\subseteq\mathbb{R}$, $$\mathcal{H}^f\Big(B\cap \left\{x\in [0,1]:0\leqslant x-x_n< r_n\mbox{ for infinitely many $n$}\right\}\Big)=\infty.$$
\end{proposition}

With the use of a tiny variant of the $K_{G,B}$ Lemma \cite[Lemma 5]{BV06}, our variant of the mass transference principle is proved, {\em mutatis mutandis}, as Proposition \ref{p3} (see Beresnevich and Velani \cite[Theorem 2]{BV06}), thus we only present a variant of  the $K_{G,B}$ Lemma. 


%
For a subset $\mathcal{K}\subseteq\{[x_n, x_n+r_n):n\geqslant 1\}$, we define $$
\mathcal{K}^f:=\left\{B^f(x_n,r_n): [x_n,x_n+r_n)\in \mathcal{K}\right\},$$ where $B^f(x_n,r_n)$ denotes the ball of radius $f(r_n)$ centred at $x_n$.

\begin{lemma}[A variant of $K_{G,B}$ Lemma]\label{KGB} Assume that the equation in (\ref{f4}) holds and let $B$ be a ball in $[0,1]$. For any $G\geqslant 1$, there exists a subset $K_{G,B}\subseteq\{[x_n, x_n+r_n)\}_{n\geqslant G}$ such that
the elements of $K_{G,B}^f$ are disjoint, inside $B$ and $$
\sum_{L\in K_{G,B}}f(r_L)\geqslant \frac{|B|}{20},
$$ where $r_{L}$ denotes the radius of the ball $L$.
\end{lemma}

\begin{proof}
The elements in $K_{G,B}$ here are nothing but half of the balls in Beresnevich and Velani's original $K_{G,B}$-Lemma \cite[Lemma 5]{BV06}.

\end{proof}

\begin{remark} Note that while we state our variant of the mass transference principle only in the one-dimensional case, it is still valid for higher dimensions.
\end{remark}

\section{Dichotomy laws for $\beta$-dynamical systems}

In this section, we establish the Jarn\'ik-type dichotomy laws of Theorems \ref{thm1} and~\ref{thm2}.

\begin{proof}[Proof of Theorem \ref{thm1}]
We consider, separately, the cases of convergence and divergence of the series \begin{equation}\label{bfb}\sum_{n\geqslant 1}\beta^n f\left(\frac{\Psi(n)}{\beta^n}\right).\end{equation}

Suppose that the series \eqref{bfb} converges. We begin by writing the set $W_y(T_\beta,\Psi)$ in a way that reflects its limsup nature. To do this, for any $\bar{\epsilon}=(\epsilon_1,\ldots,\epsilon_n)\in D_{\beta,n}$, we define \begin{equation}\label{yn}
y_n(\bar{\epsilon})=\frac{\epsilon_1}{\beta}+\cdots+\frac{\epsilon_n}{\beta^n}+\frac{y}{\beta^n}.\end{equation} We then have
\begin{align}
W_y(T_{\beta},\Psi)
\nonumber&=\bigcap_{N\geqslant 1}\bigcup_{n\geqslant N}\left\{x\in [0,1]: |T_\beta^n x-y|<\Psi(n)\right\}\\
\nonumber&=\bigcap_{N\geqslant 1}\bigcup_{\substack{n\geqslant N\\ \bar{\epsilon}\in D_{\beta,n}}}\left\{x\in I_n(\bar{\epsilon}): |T_\beta^n x-y|<\Psi(n)\right\}\\
\label{Inynyn}&=\bigcap_{N\geqslant 1}\bigcup_{\substack{n\geqslant N\\ \bar{\epsilon}\in D_{\beta,n}}}\left\{I_n(\bar{\epsilon})\cap \left(y_n(\bar{\epsilon})-\frac{\Psi(n)}{\beta^n},y_n(\bar{\epsilon})+  \frac{\Psi(n)}{\beta^n}\right)\right\},\end{align}
where, to obtain the last equality, we substituted the value of $T_\beta^n x$ in terms of $\bar{\epsilon}$, which is determined by \eqref{eq5}. 
{Note that the set inside the union in \eqref{Inynyn} can be covered by two intervals each of length $\frac{\Psi(n)}{\beta^n}$, thus} along with the definition of Hausdorff measure, the quantity $\mathcal{H}^f\left(W_y(T_\beta, \Psi)\right)$ is bounded by $$
\liminf_{N\to\infty}\sum_{n\geqslant N}\sum_{\bar{\epsilon}\in D_{\beta,n}}2\cdot f\left(\frac{\Psi(n)}{\beta^n}\right)
\leqslant \frac{2\beta}{\beta-1}\liminf_{N\to\infty}\sum_{n\geqslant N}\beta^n f\left(\frac{\Psi(n)}{\beta^n}\right)=0,
$$ where for the second inequality, we used $\#D_{\beta,n}\leqslant \beta^{n+1}/(\beta-1)$ as given in \eqref{lem1}.

Now suppose that the series \eqref{bfb} diverges. If $\{f(r)/r:r>0\}$ is bounded, $\mathcal{H}^f= c\cdot\mathfrak{L}$ for a constant $c>0$. Then Philipp's result applies. So, we assume that $f(r)/r\to \infty$ as $r\to 0$. Moreover, instead of studying the set $W_y(T_\beta, \Psi)$ directly, we consider the sets $W_y'(T_\beta, \Psi)$ (as defined in Proposition \ref{p1}) and $$W_y'(T_\beta, \Psi,f)=\left\{x\in [0,1]: 0\leqslant T_\beta^n x-y<\beta^n f\left(\frac{\Psi(n)}{\beta^n}\right)\mbox{ for infinitely many $n$}\right\}.$$ Because of our divergence assumption on the series \eqref{bfb}, by Proposition \ref{p1}, we have $\mathfrak{L}\left(W_y'(T_\beta, \Psi,f)\right)=1$.

Similar to the above, we may write
$$W_y'(T_\beta, \Psi)=\bigcap_{N\geqslant 1}\bigcup_{\substack{n\geqslant N\\\bar{\epsilon}\in D_{\beta,n}}}\left\{I_n(\bar{\epsilon})\cap \left[y_n(\bar{\epsilon}),y_n(\bar{\epsilon})+\frac{\Psi(n)}{\beta^n}\right)\right\}$$ and $$W_y'(T_\beta, \Psi,f):=\bigcap_{N\geqslant 1}\bigcup_{\substack{n\geqslant N\\\bar{\epsilon}\in D_{\beta,n}}}\left\{I_n(\bar{\epsilon})\cap \left[y_n(\bar{\epsilon}),y_n(\bar{\epsilon})+f\left(\frac{\Psi(n)}{\beta^n}\right)\right)\right\},$$ where $y_n(\bar{\epsilon})$ is as given in \eqref{yn}.

Since $y_n(\bar{\epsilon})$ is larger than the left endpoint of $I_n(\bar{\epsilon})$, we have \begin{align*}I_n(\bar{\epsilon})\cap \left[y_n(\bar{\epsilon}),y_n(\bar{\epsilon})+\frac{\Psi(n)}{\beta^n}\right)&=\left\{x\in [0,1]: 0\leqslant x-y_n(\bar{\epsilon})< r_n(\bar{\epsilon})\right\}\\ &=[y_n(\bar{\epsilon}), y_n(\bar{\epsilon})+r_n(\bar{\epsilon}))\end{align*} for some $r_n(\bar{\epsilon})\geqslant 0$ (we take $r_n(\bar{\epsilon})=0$ if the set is empty). So, \begin{equation}\label{2}W_y'(T_\beta, \Psi)=\bigcap_{N\geqslant 1}\bigcup_{\substack{n\geqslant N\\\bar{\epsilon}\in D_{\beta,n}}}\left\{x\in [0,1]: 0\leqslant x-y_n(\bar{\epsilon})< r_n(\bar{\epsilon})\right\}.\end{equation}

Similarly, \begin{align*}I_n(\bar{\epsilon})\cap \left[y_n(\bar{\epsilon}),y_n(\bar{\epsilon})+f\left(\frac{\Psi(n)}{\beta^n}\right)\right) &=\left\{x\in [0,1]: 0\leqslant x-y_n(\bar{\epsilon})< t_n(\bar{\epsilon})\right\}\\ &=[y_n(\bar{\epsilon}), y_n(\bar{\epsilon})+t_n(\bar{\epsilon}))\end{align*} for some $t_n(\bar{\epsilon})\geqslant 0$ (we take $t_n(\bar{\epsilon})=0$ if the set is empty), and so $$W_y'(T_\beta, \Psi,f)=\bigcap_{N\geqslant 1}\bigcup_{\substack{n\geqslant N\\\bar{\epsilon}\in D_{\beta,n}}}\left\{x\in [0,1]: 0\leqslant x-y_n(\bar{\epsilon})< t_n(\bar{\epsilon})\right\}.$$

We claim that for sufficiently large $n$, \begin{equation}\label{15} f(r_n(\bar{\epsilon}))\geqslant t_n(\bar{\epsilon}),\end{equation} for any $\bar{\epsilon}\in D_{\beta,n}$. To see that this is indeed the case, let $b$ be the right endpoint of $I_n(\bar{\epsilon})$. Then $$y_n(\bar{\epsilon})+r_n(\bar{\epsilon})=\min\left\{b, y_n(\bar{\epsilon})+\frac{\Psi(n)}{\beta^n}\right\}$$ and $$y_n(\bar{\epsilon})+t_n(\bar{\epsilon})=\min\left\{b, y_n(\bar{\epsilon})+f\left(\frac{\Psi(n)}{\beta^n}\right)\right\},$$ so that $$
r_n(\bar{\epsilon})=\min\left\{b-y_n(\bar{\epsilon}), \frac{\Psi(n)}{\beta^n}\right\}\quad\mbox{and}\quad t_n(\bar{\epsilon})=\min\left\{b-y_n(\bar{\epsilon}), f\left(\frac{\Psi(n)}{\beta^n}\right)\right\}.$$ Now if $t_n(\bar{\epsilon})=0$, there is nothing to prove. When $t_n(\bar{\epsilon})>0$, we have $y_n(\bar{\epsilon})<b$, thus we have $r_n(\bar{\epsilon})>0$ as well. Since $r^{-1}f(r)\to\infty$ as $r\to 0$ and $0\leqslant b-y_n(\bar{\epsilon})\leqslant \beta^{-n}$, we have that $f(b-y_n(\bar{\epsilon}))\geqslant b-y_n(\bar{\epsilon})$ when $n$ is sufficiently large, which proves the claim. Note also that $0\leqslant r_n(\bar{\epsilon})\leqslant \beta^{-n}$, so that $r_n(\bar{\epsilon})\to 0$ as $n\to\infty$.

Having established \eqref{15}, the limsup set\begin{equation}
\label{Wfrn} \bigcap_{N\geqslant 1}\bigcup_{\substack{n\geqslant N\\ \bar{\epsilon}\in D_{\beta,n}}}\left\{x\in [0,1]: 0\leqslant x-y_n(\bar{\epsilon})<f(r_n(\bar{\epsilon}))\right\}\end{equation}
contains the set $W_y'(T_\beta, \Psi, f)$, so it is also of full Lebesgue measure. Then a direct application of our variant of the mass transference principle (Proposition \ref{p4}) yields $\mathcal{H}^f(W_y'(T_\beta, \Psi))=\infty.$ As $W_y'(T_\beta, \Psi)\subseteq W_y(T_\beta, \Psi)$, the result follows.
\end{proof}

\begin{remark}
Loosely speaking, $y_n(\bar{\epsilon})$ is the inverse $T_\beta^{-n}y$ of $y$ in the cylinder $I_n(\bar{\epsilon})$. But, this is not always the case. When the length of the interval $I_n(\bar{\epsilon})$ is strictly less than $\beta^{-n}$, for any $y>\beta^n|I_n(\bar{\epsilon})|$, there does not exist $x\in I_n(\bar{\epsilon})$, such that $T_\beta^n x=y$. So, it is possible that $$I_n(\bar{\epsilon})\cap \left(y_n(\bar{\epsilon})-\frac{\Psi(n)}{\beta^n},y_n(\bar{\epsilon})+\frac{\Psi(n)}{\beta^n}\right)$$ is empty. In view of this, the magnified set $$I_n(\bar{\epsilon})\cap \left(y_n(\bar{\epsilon})-f\Big(\frac{\Psi(n)}{\beta^n}\Big),y_n(\bar{\epsilon})+f\Big(\frac{\Psi(n)}{\beta^n}\Big)\right)$$
may contribute  to the Lebesgue measure of the limsup set $$\left\{x\in [0,1]: |T_\beta^n x-y|<\beta^n f\left(\frac{\Psi(n)}{\beta^n}\right)\mbox{ for infinitely many $n$}\right\},$$ while the shrunk set $$I_n(\bar{\epsilon})\cap \left(y_n(\bar{\epsilon})-\frac{\Psi(n)}{\beta^n},y_n(\bar{\epsilon})+\frac{\Psi(n)}{\beta^n}\right)$$ may contribute nothing to the size of $W_y(T, \Psi)$.
This is the main reason we consider the subset $W'_y(T, \Psi)$, instead of $W_y(T, \Psi)$.

%

\end{remark}

We continue this section with the proof of Theorem \ref{thm2}.

\begin{proof}[Proof of Theorem \ref{thm2}]
We consider, separately, the cases of convergence and divergence of the series \begin{equation}\label{PbfPb}\sum_{n\geqslant 1} g\left(\frac{\Psi(n)}{\beta^n}\right)\frac{\beta^{2n}}{\Psi(n)}.\end{equation}

As in the proof of Theorem \ref{thm1}, we begin by writing the limsup version of the set $W(T_\beta,\Psi)$. To do this, for each $n\in\mathbb{N}$, set \begin{equation}\label{eq2}W_n(T_\beta,\Psi):=\left\{(x, y)\in [0,1]^2:|T_\beta^n x-y|<\Psi(n)\right\}.\end{equation} Then \begin{equation*}
W(T_\beta,\Psi)=\limsup_{n\to\infty}W_n(T_\beta,\Psi)=\bigcap_{N\geqslant 1}\bigcup_{n\geqslant N}W_n(T_\beta,\Psi).
\end{equation*}

On the other hand, the interval $[0,1]$, which corresponds to the doubly metric parameter $y$, can also be written as a union of intervals $$J_n(i)=\left[\frac{i\Psi(n)}{\beta^n}, \frac{(i+1)\Psi(n)}{\beta^n}\right]\cap[0,1]$$ over all $0\leqslant i\leqslant \lfloor\beta^n/\Psi(n)\rfloor$; that is, \begin{equation}\label{eq3} [0,1]=\bigcup_{0\leqslant i\leqslant \lfloor\beta^n/\Psi(n)\rfloor}J_n(i).\end{equation}

Combining \eqref{eq1} and \eqref{eq3}, we have \begin{equation*}[0,1]^2=\bigcup_{\substack{\bar{\epsilon}\in D_{\beta,n}\\ 0\leqslant i\leqslant \lfloor\beta^n/\Psi(n)\rfloor}}I_n(\bar{\epsilon})\times J_n(i).\end{equation*} This, when combined with \eqref{eq2}, gives \begin{equation*}\label{eq4} W_n(T_\beta,\Psi)=\bigcup_{\substack{\bar{\epsilon}\in D_{\beta,n}\\ 0\leqslant i\leqslant \lfloor\beta^n/\Psi(n)\rfloor}}\left\{(x, y)\in I_n(\bar{\epsilon})\times J_n(i):|T_\beta^n x-y|<\Psi(n)\right\}.\end{equation*} Note that, given any $\bar{\epsilon}\in D_{\beta,n}$ and $0\leqslant i\leqslant \lfloor\beta^n/\Psi(n)\rfloor$, 
{if $$(x, y)\in \left\{I_n(\bar{\epsilon})\times J_n(i):|T_\beta^n x-y|<\Psi(n)\right\}$$}then 
$$\left|T^n_\beta x-\frac{i\Psi(n)}{\beta^n}\right|\leqslant \left|T^n_\beta x-y\right|+\left|y-\frac{i\Psi(n)}{\beta^n}\right| \leqslant \Psi(n)+\frac{\Psi(n)}{\beta^n}<2\Psi(n).$$ 
Set $$z_n=\frac{\epsilon_1}{\beta}+\cdots+\frac{\epsilon_n}{\beta^n}+\frac{i\Psi(n)}{\beta^n}.$$ Then, similar to the convergence part of the proof of Theorem \ref{thm1}, it follows that
$$x\in\left(z_n-\frac{2\Psi(n)}{\beta^n}, z_n+\frac{2\Psi(n)}{\beta^n}\right).$$ Thus $$W_n(T_\beta,\Psi)\subseteq \bigcup_{\substack{\bar{\epsilon}\in D_{\beta,n}\\ 0\leqslant i\leqslant \lfloor\beta^n/\Psi(n)\rfloor}} \left(z_n-\frac{2\Psi(n)}{\beta^n}, z_n+\frac{2\Psi(n)}{\beta^n}\right) \times J_n(i).$$
As a result, {for $N\geqslant 1$} , the family of rectangles $$\bigcup_{\substack{n\geqslant N\\ \bar{\epsilon}\in D_{\beta,n}\\ 0\leqslant i\leqslant \lfloor\beta^n/\Psi(n)\rfloor}} \left(z_n-\frac{2\Psi(n)}{\beta^n}, z_n+\frac{2\Psi(n)}{\beta^n}\right) \times J_n(i)$$ is a cover for the set $W_n(T_\beta,\Psi)$. Moreover, each of these rectangles can be covered by $64$ cubes with diameter $\Psi(n)/\beta^n$.
So, by the definition of Hausdorff measure, it follows, using \eqref{lem1}, that $\mathcal{H}^g(W(T_\beta,\Psi))$ is bounded by a constant times
\begin{equation}\label{limgPb}\liminf_{N\to \infty}\sum_{n\geqslant N}\sum_{\bar{\epsilon}\in D_{\beta,n}} \sum_{0\leqslant i\leqslant \lfloor\beta^n/\Psi(n)\rfloor}g\left(\frac{\Psi(n)}{\beta^n}\right) 
\ll \liminf_{N\to\infty}\sum_{n\geqslant N} g\left(\frac{\Psi(n)}{\beta^n}\right)\frac{\beta^{2n}}{\Psi(n)}.\end{equation}

If the series in \eqref{PbfPb} converges, it follows from \eqref{limgPb} that $\mathcal{H}^g(W(T_\beta,\Psi))=0.$

Now suppose that the series in \eqref{PbfPb} diverges, and set $f(x)=x^{-1}g(x)$. Then $$
\sum_{n\geqslant 1} f\left(\frac{\Psi(n)}{\beta^n}\right)\beta^n=\sum_{n\geqslant 1} g\left(\frac{\Psi(n)}{\beta^n}\right)\frac{\beta^{2n}}{\Psi(n)}=\infty.
$$By Theorem~\ref{thm1}, we have shown that for any fixed $y\in [0,1]$, $\mathcal{H}^f(W_y(T_\beta, \Psi))=\mathcal{H}^f([0,1]).$ Then, appealing to the slicing lemma (Lemma \ref{slicing}), we have immediately that $\mathcal{H}^g(W(T_\beta, \Psi))=\mathcal{H}^g([0,1]^2).$
\end{proof}

\section{Concluding remarks and questions}

In this paper, we considered the $\beta$-expansion of real numbers for real values of $\beta>1$. Of course, it is very interesting to restrict $\beta$ to the set of algebraic numbers, and even more interesting to restrict both $\beta$ and the numbers we are expanding to be algebraic numbers. To the best of our knowledge, the only known consideration of this case was made by Bugeaud \cite{B2009} who gave results concerning the number of digit changes in the $\beta$-expansion of algebraic numbers (where $\beta$ is algebraic). The questions of Hausdorff dimension and Hausdorff measure, as well as the possibility of any dichotomy law, are still open.

As we mentioned in the Introduction, the set $W_y(T_\beta, \Psi)$ is closely related to the distribution of the preimage $${\rm Pre}(T_\beta, y)=\left\{T_\beta^{-n}y, n\in \mathbb{N}\right\}$$ of $y$. Given a general compact metric space $(X, T)$, one can consider the size of $$W(T, \Psi)=\left\{x\in X: |T^n x-y|<\Psi(n)\mbox{ for infinitely many $n$}\right\}.$$ If ${\rm Pre}(T, y)$ is {\em well}-distributed,
one hopes to derive some information on the Hausdorff dimension and Hausdorff measure of $W(T, \Psi)$. But how {\em well} should ${\rm Pre}(T, y)$ be distributed sufficient to get the dimension or even measure? Can one give a precise characterisation on this requirement? To which extent can a general system fulfill the required conditions?

\medskip

\noindent\emph{Acknowledgements.} MH would like to thank Stephen Harrap for many useful discussions and to members of the Departmetn of Mathematics at Huazhong University for their hospitality.

\bibliographystyle{amsplain}

\begin{thebibliography}{10}

\bibitem{BaBV13}
D.~Badziahin, V.~Beresnevich, and S.~Velani, \emph{Inhomogeneous theory of dual
  {D}iophantine approximation on manifolds}, Adv. Math. \textbf{232} (2013),
  1--35. \MR{2989975}

\bibitem{BaHH15}
D.~Badziahin, S.~Harrap, and M.~Hussain, \emph{An inhomogeneous {J}arn\'ik type
  theorem for planar curves}, {\tt arXiv:1503.04981} (2015), preprint.

\bibitem{BaS}
A.~Baker and W. Schmidt, \emph{Diophantine approximation and
  {H}ausdorff dimension}, Proc. London Math. Soc. (3) \textbf{21} (1970),
  1--11. \MR{0271033 (42 \#5916)}

\bibitem{BerDV}
\bysame, \emph{Measure theoretic laws for lim sup sets}, Mem. Amer. Math. Soc.
  \textbf{179} (2006), no.~846, x+91. \MR{2184760 (2007d:11086)}

\bibitem{BV06}
V. Beresnevich and S. Velani, \emph{A mass transference principle and
  the {D}uffin-{S}chaeffer conjecture for {H}ausdorff measures}, Ann. of Math.
  (2) \textbf{164} (2006), no.~3, 971--992. \MR{2259250 (2008a:11090)}

\bibitem{BV06_IMRN}
\bysame, \emph{Schmidt's theorem, {H}ausdorff measures, and slicing}, Int.
  Math. Res. Not. (2006), Art. ID 48794, 24. \MR{2264714 (2007h:11090)}

\bibitem{BeDo_99}
V.~I. Bernik and M.~M. Dodson, \emph{Metric {D}iophantine approximation on
  manifolds}, Cambridge Tracts in Mathematics, vol. 137, Cambridge University
  Press, Cambridge, 1999. \MR{1727177 (2001h:11091)}

\bibitem{B2009}
Y. Bugeaud, \emph{On the {$\beta$}-expansion of an algebraic number in an
  algebraic base {$\beta$}}, Integers \textbf{9} (2009), A20, 215--226.
  \MR{2534910 (2010i:11101)}

\bibitem{BuWa14}
Y. Bugeaud and B. Wang, \emph{Distribution of full cylinders and the
  {D}iophantine properties of the orbits in {$\beta$}-expansions}, J. Fractal
  Geom. \textbf{1} (2014), no.~2, 221--241. \MR{3230505}

\bibitem{ChK}
N.~Chernov and D.~Kleinbock, \emph{Dynamical {B}orel-{C}antelli lemmas for
  {G}ibbs measures}, Israel J. Math. \textbf{122} (2001), 1--27. \MR{1826488
  (2002h:37003)}

\bibitem{Do97}
M.~M. Dodson, \emph{A note on metric inhomogeneous {D}iophantine
  approximation}, J. Austral. Math. Soc. Ser. A \textbf{62} (1997), no.~2,
  175--185. \MR{1433207 (98b:11085)}

\bibitem{DodV}
M.~M. Dodson, B.~P. Rynne, and J.~A.~G. Vickers, \emph{Diophantine
  approximation and a lower bound for {H}ausdorff dimension}, Mathematika
  \textbf{37} (1990), no.~1, 59--73. \MR{1067887 (91i:11098)}

\bibitem{F_03}
K. Falconer, \emph{Fractal geometry}, second ed., John Wiley \& Sons,
  Inc., Hoboken, NJ, 2003, Mathematical foundations and applications.
  \MR{2118797 (2006b:28001)}

\bibitem{FanWa}
A.-H. Fan and B. Wang, \emph{On the lengths of basic intervals in beta
  expansions}, Nonlinearity \textbf{25} (2012), no.~5, 1329--1343. \MR{2914142}

\bibitem{GeLu}
Yuehua Ge and Fan L{\"u}, \emph{A note on inhomogeneous {D}iophantine
  approximation in beta-dynamical system}, Bull. Aust. Math. Soc. \textbf{91}
  (2015), no.~1, 34--40. \MR{3294256}

\bibitem{Har}
G. Harman, \emph{Metric number theory}, London Mathematical Society
  Monographs. New Series, vol.~18, The Clarendon Press, Oxford University
  Press, New York, 1998. \MR{1672558 (99k:11112)}

\bibitem{HiV95}
R. Hill and S. Velani, \emph{The ergodic theory of shrinking
  targets}, Invent. Math. \textbf{119} (1995), no.~1, 175--198. \MR{1309976
  (96e:58088)}

\bibitem{HiV99}
\bysame, \emph{The shrinking target problem for matrix transformations of
  tori}, J. London Math. Soc. (2) \textbf{60} (1999), no.~2, 381--398.
  \MR{1724857 (2000i:37003)}

\bibitem{HiV02}
\bysame, \emph{A zero-infinity law for well-approximable points in {J}ulia
  sets}, Ergodic Theory Dynam. Systems \textbf{22} (2002), no.~6, 1773--1782.
  \MR{1944403 (2003m:37065)}

\bibitem{Hof}
F. Hofbauer, \emph{{$\beta $}-shifts have unique maximal measure}, Monatsh.
  Math. \textbf{85} (1978), no.~3, 189--198. \MR{0492180 (58 \#11326)}

\bibitem{HK13_2}
M.~Hussain and S.~Kristensen, \emph{Metrical results on systems of small linear
  forms}, Int. J. Number Theory \textbf{9} (2013), no.~3, 769--782.
  \MR{3043613}

\bibitem{HK15}
\bysame, \emph{Metrical theorems on systems of small inhomogeneous linear
  forms}, {\tt arXiv:1406.3930} (2015), preprint.

\bibitem{HL13}
M. Hussain and J. Levesley, \emph{The metrical theory of simultaneously
  small linear forms}, Funct. Approx. Comment. Math. \textbf{48} (2013),
  no.~part 2, 167--181. \MR{3100138}

\bibitem{Kim} M. Fuchs and D. Kim, \emph{On Kurzweil's 0-1 Law in Inhomogeneous Diophantine Approximation},
  arXiv:1501.04714, 2015.

\bibitem{LiW}
B. Li, B. Wang, J. Wu, and J. Xu, \emph{The shrinking target problem
  in the dynamical system of continued fractions}, Proc. Lond. Math. Soc. (3)
  \textbf{108} (2014), no.~1, 159--186. \MR{3162824}

\bibitem{Pa60}
W.~Parry, \emph{On the {$\beta $}-expansions of real numbers}, Acta Math. Acad.
  Sci. Hungar. \textbf{11} (1960), 401--416. \MR{0142719 (26 \#288)}

\bibitem{PeSch}
T. Persson and J. Schmeling, \emph{Dyadic {D}iophantine approximation
  and {K}atok's horseshoe approximation}, Acta Arith. \textbf{132} (2008),
  no.~3, 205--230. \MR{2403650 (2009c:11111)}

\bibitem{Phi}
W. Philipp, \emph{Some metrical theorems in number theory}, Pacific J.
  Math. \textbf{20} (1967), 109--127. \MR{0205930 (34 \#5755)}

\bibitem{Rev}
H. Reeve, \emph{Shrinking targets for countable {M}arkov maps}, {\tt
  arXiv:1107.4736} (2011), preprint.

\bibitem{Re57}
A.~R{\'e}nyi, \emph{Representations for real numbers and their ergodic
  properties}, Acta Math. Acad. Sci. Hungar \textbf{8} (1957), 477--493.
  \MR{0097374 (20 \#3843)}

\bibitem{Schme}
J. Schmeling, \emph{Symbolic dynamics for {$\beta$}-shifts and
  self-normal numbers}, Ergodic Theory Dynam. Systems \textbf{17} (1997),
  no.~3, 675--694. \MR{1452189 (98c:11080)}

\bibitem{ShWa}
L. Shen and B. Wang, \emph{Shrinking target problems for beta-dynamical
  system}, Sci. China Math. \textbf{56} (2013), no.~1, 91--104. \MR{3016585}

\bibitem{TaW}
B. Tan and B. Wang, \emph{Quantitative recurrence properties for
  beta-dynamical system}, Adv. Math. \textbf{228} (2011), no.~4, 2071--2097.
  \MR{2836114}

\bibitem{Ts08}
J. Tseng, \emph{On circle rotations and the shrinking target properties},
  Discrete Contin. Dyn. Syst. \textbf{20} (2008), no.~4, 1111--1122.
  \MR{2379490 (2010a:37080)}

\bibitem{Urb}
M. Urba{\'n}ski, \emph{Diophantine analysis of conformal iterated function
  systems}, Monatsh. Math. \textbf{137} (2002), no.~4, 325--340. \MR{1947918
  (2004j:37085)}

\end{thebibliography}
\providecommand{\bysame}{\leavevmode\hbox to3em{\hrulefill}\thinspace}
\providecommand{\MR}{\relax\ifhmode\unskip\space\fi MR }
\providecommand{\MRhref}[2]{%
  \href{http://www.ams.org/mathscinet-getitem?mr=#1}{#2}
}
\providecommand{\href}[2]{#2}


\end{document}